\documentclass[12pt]{amsart}
\usepackage{amsmath, amssymb, amsthm}
\usepackage{enumerate}
\usepackage{bbm}
\usepackage{fancyhdr}

\theoremstyle{plain}
\newtheorem{theorem}{Theorem}[section]
\newtheorem{prop}{Proposition}[section]
\newtheorem{lem}{Lemma}[section]

\theoremstyle{definition}

\numberwithin{equation}{section}
\newcommand{\indicator}[1]{\mathbbm{1}_{ {#1} }}

\DeclareMathOperator*{\R}{\mathbb{R}}

\title[Weighted Local  Estimates for Fractional Operators]{Weighted Local  Estimates for Fractional Type Operators}
\author[A. Torchinsky]{Alberto Torchinsky}

\begin{document}

\begin{abstract}
In this note we prove the  estimate $M^{\sharp}_{0,s}(Tf)(x) \le c\,M_\gamma f(x)$ for
general fractional type operators $T$, where $M^{\sharp}_{0,s}$ is the local sharp maximal
function and $M_\gamma$ the fractional maximal function, as well as a local version of this estimate.
This allows us to
 express the local weighted   control of $Tf$ by $M_\gamma f$.
Similar estimates hold for $T$ replaced by
 fractional type operators with kernels satisfying H\"{o}rmander-type conditions or integral
 operators with homogeneous kernels, and $M_\gamma $ replaced by an appropriate maximal
 function $M_T$.
 We also prove
 two-weight, $L^p_v$-$L^q_w$   estimates for the fractional type operators described above for $1<p< q<\infty$
 and a range of $q$.
 The local  nature of the  estimates leads to results involving generalized
 Orlicz-Campanato  and Orlicz-Morrey spaces.
\end{abstract}

\maketitle

\setcounter{section}{1}
\section*{Introduction}

The purpose of this paper is to establish   that
much like  the Hardy-Littlewood maximal function controls the Calder\'{o}n-Zygmund singular integral operators \cite{PT},
integral operators of fractional type are controlled by   fractional maximal functions.
 Muckenhoupt and Wheeden formulated this principle for the Riesz potentials in the weighted setting
as follows  \cite{MuckWhee}.
For $0<\gamma <1$   let
\[ I_\gamma f(x)=\int_{\R^n} \frac{ f(y)}{|x-y|^{n(1-\gamma)}}\,dy \]
denote  the Riesz potential of order $\gamma$, and
\[ M_\gamma f(x) =\sup_{x\in Q}\frac1{|Q|^{1-\gamma}}\int_{Q} |f(y)|\,dy
\]
the fractional maximal function of order $\gamma$.
Then, if   $w$ is in $A_\infty$
and  $0 < q <\infty$, there is a constant $c$  independent of $f$  such that
\begin{equation}
\int_{\R^n} |I_\gamma f(x)|^q\,w(x)\,dx\le  c\,
\int_{\R^n} M_\gamma f(x)^q\, w(x)\,dx\,.
\end{equation}

In fact,    there is a subtle interplay between these two
operators. It is  expressed by the readily verified pointwise inequality
\begin{equation}
 M_\gamma f(x)\le c\, I_\gamma(| f|)(x)\,,
\end{equation}
where $c$  depends on the dimension $n$   and   $\gamma$, and the pointwise inequality
\begin{equation}
 M^\sharp(I_\gamma f)(x) \le c\, M_\gamma f(x)\,,
\end{equation}
where $M^\sharp$ denotes the sharp maximal function and $c$ is independent of $f$ and $x$,
established  by Adams \cite{Ad}.
Note that (1.2) gives  the equivalence of the norms in (1.1),  and  that
under appropriate conditions   (1.3) implies  (1.1).

Gogatishvili and Mustafayev   \cite{GogaMus} observed  that (1.3) also implies
that  for an arbitrary cube $Q_0$ and $1<q<\infty$,
\begin{equation}  \int_{Q_0}  |I_\gamma f(x) - (I_\gamma f)_{Q_0}|^q\,dx
\le c \int_{Q_0}   M_\gamma f(x)^q\,dx\,,
\end{equation}
where $(I_\gamma f)_{Q_0}$ denotes the average of $I_\gamma f$ over   $Q_0$ and $c$ is independent of $f$ and $Q_0$.
Rakotondratsimba obtained similar weighted local
inequalities for a local version of the Riesz potentials \cite{Rako}.
The estimate (1.4)  allows for the comparison of the norms of Riesz potentials and fractional maximal functions
in Morrey-type spaces  \cite{GogaMus}.

Here   we consider  general  fractional type  operators  given by
\[Tf(x) = \int_{\R^n}k(x,y)f(y)\,dy \]
where for some fixed $0 < \gamma < 1$,
there exists a  positive constant  $c$ such that   for every cube $Q$,
\begin{equation*}
|k(x,y) - k(x',y)| \le c \,\frac{1}{|x-y|^{n(1-\gamma)}}\, \omega \Big ( \frac{|x-x'|}{|x-y|} \Big )
\end{equation*}
whenever $x,x' \in Q$ and $y \in (2Q)^c$,
where $\omega(t)$ is a nondecreasing function on $(0,\infty)$ that satisfies an appropriate Dini-type condition.

In the first part of this paper we prove that  in particular,
   if  $T$ is of weak-type $  (1, 1/(1-\gamma ))$,      we have
the local pointwise   estimate
\begin{equation}
M^{\sharp}_{0,s,Q_0}  (Tf )(x) \le c \sup_{x \in Q, Q \subset Q_0} \inf_{y \in Q} M _\gamma f(y)\,,
\end{equation}
where $M^{\sharp}_{0,s,Q_0}$  denotes the local sharp maximal function restricted to the cube $Q_0$,
and when $Q_0 =\R^n$,
\begin{equation} M^\sharp_{0,s}(Tf)(x)\le c\, M_\gamma f(x)
\end{equation}
where $M^{\sharp}_{0,s}$  denotes the local sharp maximal function.

The weighted local estimates  follow readily from (1.5) and (1.6).
By a weight  we mean a nonnegative locally integrable function $w$, and
 we say that  a  continuous function $\Phi$ satisfies condition
 $C$ if it is   increasing on $[0, \infty)$ with  $\Phi(0) = 0$ and $\Phi(2t) \le c\,\Phi(t)$, all $t>0$.
 Then   by Theorem 5.1 in \cite{PT}, (1.5) gives that
 if $\Phi$ satisfies condition $C$ and  $w$ is a  weight, for every cube $Q_0$ of $\R^n$,
\begin{equation}  \int_{Q_0} \Phi   (|T f(x) - m_{Tf}(t, Q_0)|  )\,w(x)\,dx
\le c \int_{Q_0} \Phi  ( M_\gamma f(x)   )\,v(x)\,dx\,,
\end{equation}
 where  $m_{Tf}(t, Q_0)$ is the (maximal) median  of $Tf$ with parameter $t$,
 $v=w$ when $w\in A_\infty$ and $v=M_r(w)$, the
 Hardy-Littlewood maximal function of order $r$ of $w$,
 $1<r<\infty$, when $w$ is an arbitrary weight, and $c$ is independent of $Q_0$ and $f$.

  Furthermore, if
$\lim_{Q_0\to\R^n} m_{Tf}(t,Q_0)=0$,
\begin{equation}
\int_{\R^n} \Phi   ( |Tf(x) | ) \,w(x)\,dx \le c \int_{\R^n}   \Phi ( M_\gamma  f(x) ) \, v(x)\, dx\,.
\end{equation}
And, (1.7) and (1.8) also hold  for appropriate non-$A_\infty$ weights $w$ and $ v$.

Thus, weighted local estimates hold for fractional type operators
and the integral inequality (1.8) also holds for those operators they majorize,
including  Marcinkiewicz integrals, fractional powers of  analytic semigroups,
and   Schr\"odinger type operators \cite{GulShu0, GulShu}.
And, as illustrated below in the case
of fractional type operators with kernels satisfying H\"ormander-type conditions,
and integral operators with homogeneous kernels, our approach applies in other instances as well.

Next we take a closer look at two-weight, $L^p_v-L^q_w$  specific   inequalities. The
question of determining weights $(w,v)$ so that the Riesz potentials map
$L^p_v$ continuously into $L^p_w$  was addressed by P\'erez \cite{CPIUMJ, CZSobo}
and continues to attract considerable attention.
When the computability of the conditions on the weights is a concern,  interesting results are
proved and referenced, for instance,  in \cite{Rako}.

To deal with fractional type operators in the two-weight context   we rely on the sharper
local median decomposition produced in \cite{PT};  results in this direction were anticipated in \cite{FujiiFrac}.
The Orlicz ``bump'' conditions of  P\'erez \cite{Perez1990}
and  Cruz-Uribe and Moen \cite{CUKM} and a technique of Lerner \cite{Lerner2010},
give then the estimate for   these operators, including
those of Dini type, or with kernels satisfying a H\"ormander-type condition,
from $L^p_v(\R^n)$ into $L^q_w(\R^n)$ for $1<p< q<\infty$ and  a range of  $ q$.

Finally, the local estimates   are well suited to
the generalized  Orlicz-Morrey spaces $\mathcal M^{\Phi, \phi} $
and generalized Orlicz-Campanato spaces  $\mathcal L^{\Phi, \phi} $, defined in Section 5.
Indeed, if  $T$ is a fractional type operator, from (1.7) it readily follows that for
every Young function $\Psi$   and every appropriate $\psi$,
\begin{equation}
 \|Tf\|_{\mathcal L^{\Psi, \psi} }\le c\,\|M_\gamma f\|_{\mathcal M^{\Psi, \psi} }\,.
 \end{equation}

And, concerning the continuity of $M_\gamma$ in the Orlicz-Morrey spaces we have
that for $0\le \gamma<1$, if
the Young functions $\Phi,\Psi$ are such that $\Psi^{-1}(t)=t^{-\gamma} \Phi^{-1}(t)$,  and  $\phi,\psi$ satisfy
$\sup_{l<t<\infty}  t^{n\gamma} \phi(x,t)\le c\, \psi(x,l)$,
then $M_\gamma$ maps $M^{\Phi,\phi}$ continuously into $M^{\Psi,\psi}$.

The paper is organized as follows. The
essential ingredient  in what follows, i.e.,  the estimate
 $M^{\sharp}_{0,s}(Tf)(x) \le  c\,M_{\gamma, r}f(x)$ for fractional type operators
 of weak type  $(r,r/(1-\gamma  r))$   and its local version, are done in   Section 2. We also recast similar
 estimates with $T$ replaced by fractional  type operators with kernels satisfying H\"ormander-type conditions
or integral operators with homogeneous kernels and $M$ by an appropriate maximal function $M_T$.
In Section 3 we  use these estimates   to express the local integral control of
$Tf$ in terms of $M_Tf$. In Section 4 we
   prove two-weight, $L^p_v-L^q_w$    specific estimates
for fractional type operators. And finally  in Section 5 we consider the Orlicz-Morrey and Orlicz-Campanato spaces.

Some closely related topics are not addressed here. Because we
concentrate on integral inequalities,   weak-type inequalities are not considered,
nor homogeneous spaces,
the foundation for which has been laid in \cite{SawWhee, ShiTor, TorchinskyStromberg, Rodrigo}.
And, for the various definitions or properties that the reader may find  unfamiliar,
there are many treatises in the area which may be helpful, including \cite{TorchinskyStromberg, Torch}.
It is a pleasure to acknowledge the conversations I had with J. Poelhuis concerning these matters.

\section{Pointwise Local Estimates}
In what follows we adopt the notations of \cite{MedContOsc, PT, Stromberg}.
In particular, all cubes have sides parallel to the axes. Also, for a cube $Q \subset \mathbb{R}^n$ and $0 < t < 1$, we say that
\[ m_f(t,Q) = \sup\{M : |\{y \in Q: f(y) < M\}| \le t|Q|\}\]
 is the (maximal) median of $f$ over $Q$ with parameter $t$. For a cube $Q_0 \subset \mathbb{R}^n$ and $0 < s \le 1/2$,
 the local sharp maximal function restricted to $Q_0$ of a measurable function $f$ at $x \in Q_0$ is
\begin{align*}
M^{\sharp}_{0,s,Q_0}& f(x)\\
&=  \sup_{ {x \in Q, Q \subset Q_0}}\inf_c \; \inf   \{\alpha \ge 0: |\{y \in Q: |f(y) - c| > \alpha\}| < s|Q|   \}\,,
\end{align*}
and the local sharp maximal function of a measurable function $f$ at $x \in \R^n$ is
\[M^{\sharp}_{0,s}f(x) = \sup_{ {x \in Q}}\inf_c \; \inf   \{\alpha \ge 0: |\{y \in Q: |f(y) - c| > \alpha\}| < s|Q|  \}\,.\]

Then, with the notation
\begin{equation}
m^{\sharp}_{f}(1-s,Q) = \inf_c m_{|f - c|}(1-s,Q)\,,
\end{equation}
since  by  (4.3) of \cite{MedContOsc},
$m^{\sharp}_f(1-s,Q) \sim  m_{|f - m_f(1-s,Q)|}(1-s,Q)$,
we have
\begin{align*}
M^{\sharp}_{0,s,Q_0 }f(x) &\sim \sup_{x\in Q, Q \subset Q_0} m^{\sharp}_f(1-s,Q)\\
&\sim \sup_{x\in Q, Q \subset Q_0} {m_{|f - m_f(1-s,Q)|}(1-s,Q)}\,.
\end{align*}

Next we introduce the fractional maximal functions of interest to us.
 A    function $A$ that satisfies condition $C$ which is convex and such that
 $A(t)\to \infty$ as $t \to \infty$, or, more generally,
   such that $A(t)/t\to\infty$ as $t \to\infty$, is called a Young function.
For a Young function $A$  let
\begin{equation}
\|f\|_{L^A(Q)}= \inf \Big\{\lambda>0: \frac1{|Q|}\int_Q A\Big( \frac{|f(y)|}{\lambda}\Big)\,dy\le 1\Big\}\,,
\end{equation}
and for $0\le \gamma<1$, let
\[M_{\gamma, A} f(x)= \sup_{x\in Q} |Q|^{  \gamma}\,\|  f\|_{L^A(Q)}\,.
\]
In particular, for $A(t)=t^r$, we denote $M_{\gamma, A}=M_{\gamma, r}$, and of course,
$M_\gamma =M_{\gamma, 1}\le M_{\gamma,r}$ for $1<r<\infty$. Also, when
$\gamma=0$ we drop the subscript corresponding to $\gamma$.

Finally,   observe that  there exists
a dimensional constant $c_n$ such that for every cube $Q $ in $\R^n$, if $x,x'\in Q$ and
  $y\notin 2^m Q$  for some $m\ge 1$, then
\begin{equation}  \frac{|x - x'|}{|x - y|} \le c_n \,2^{-m}\,.
\end{equation}

We then have,

\begin{theorem}
 Let $T$ be a fractional  type operator defined by
\begin{equation*}
Tf(x) =   \int_{\mathbb{R}^n} k(x,y)f(y)dy
\end{equation*}
such that for some fixed $0 < \gamma < 1$,
\begin{enumerate}
\item[\rm (1)]
there exists a constant $c> 0$ such that
\[ | k(x,y) - k(x',y)| \le c\,\frac1{|x-y|^{n(1-\gamma)}}\,  \omega\Big(\frac{|x - x'|}{|x - y|}\Big)\]
     whenever $x,x'\in Q$ and $y\in (2Q)^c$ for any cube $Q$,  where $\omega(t)$ is a nondecreasing function
     on $(0,\infty)$ such that
     \[\int_0^1 \omega(c_n t)  \, \frac{dt}{t}<\infty\,,\]
     and
\item[\rm (2)]   $T$ is of weak-type $(r,r/(1-\gamma  r))$, for some $1 \le r<\infty$.
\end{enumerate}

Then, for  $0 < s \le  1/2$, any cube $Q_0$, and $x\in Q_0$,
\begin{equation}
M^{\sharp}_{0,s,Q_0}(Tf)(x) \le c \sup_{x \in Q, Q \subset Q_0} \inf_{y \in Q} M_{\gamma,r} f(y)\,.
\end{equation}

In particular, if $Q_0 = \mathbb{R}^n$, then for all $x\in \R^n$,
\[M^{\sharp}_{0,s}(Tf)(x) \le c\, M_{\gamma,r} f(x)\,.\]
\end{theorem}

\begin{proof}
Fix a cube $Q_0 \subset \mathbb{R}^n$  and take $x \in Q_0$.
Let $Q \subset Q_0$ be a cube  centered at  $x_Q$ containing $x$.
Let $f_1 =   f \indicator{2Q}$ and $f_2 =   f - f_1$;
then by the linearity of $T$, $Tf(z)- Tf_2(x_Q)= Tf_1(z)+ Tf_2(z)-Tf_2(x_Q)$ for $z\in Q$.

We claim that there exist constants $c_1,c_2>0$ independent of $f$ and $Q$ such that
\begin{equation}
|\{z \in Q: |Tf_1(z)| >  c_1 \inf_{y \in Q}M_{\gamma,r} f(y)\}| < s\,|Q|\,,
\end{equation}
and
\begin{equation}
 \|Tf_2 - Tf_2(x_Q) \|_{L^{\infty}(Q)} \le  c_{2} \inf_{y \in Q}M_{\gamma,r} f(y)\,.
\end{equation}

We prove (2.6) first. Observe that for any $z \in Q$,  by (2.3),
\begin{align}
|Tf_2 (z) &- Tf_2(x_Q)|\notag\\
&\le  \int_{(2Q)^c}|k(z,y) - k(x_Q,y)|\, |f(y)|\,dy
\notag
\\
&\le c \sum_{m=1}^{\infty} \int_{2^{m+1}Q\setminus 2^{m}Q}  \frac1{|z-y|^{n(1-\gamma)}}\, \omega\Big(\frac{|x_Q - z|}{| y-z|}\Big)\,
|f(y)  |\, dy\notag\\
&\le c \sum_{m=1}^{\infty} \omega(c_n/ 2^{m})\, \frac1{|2^mQ|^{1-\gamma}}\int_{2^mQ} |f(y) |\,dy\notag\\
&\le c  \Big(\int_0^1 \omega(c_n t) \, \frac{dt}{t}\Big) \inf_{y \in Q}M_{\gamma,r} f(y)\,,
\end{align}
and (2.6) holds.

As for (2.5), since $T$ is of weak-type $(r, r/(1-\gamma r))$
 we have that for any $\lambda > 0$,
   \begin{align}
{\lambda^{r/(1-\gamma r)}} | \{z \in Q: &|Tf_1(z)| > \lambda\}|\notag\\
 &\le  {c} \, \Big(\int_{2Q}|f(y) |^r\,dy\Big)^{1/(1-\gamma r)}\notag
\\
 &=  {c} \,\Big( |2Q|^{\gamma} \Big(\frac1{| 2Q|}
  \int_{2Q}|f(y)  |^r\,dy \Big)^{1/r} \Big)^{{r/(1-\gamma r)}} |Q|\notag\\
 &\le c   \, \inf_{y \in Q}M_{\gamma  , r} f(y)^{r/(1-\gamma r)}\,|Q| \,,
\end{align}
and (2.5) follows by picking $\lambda =c_1 \inf_{y\in Q} M_{\gamma, r} f(y)$
for an appropriately chosen $c_1$.

 Then, with $c > \max \{c_1, c_2\}$, (2.5) and (2.6) give
\begin{align*}
|\{z \in Q&: |Tf(z) - Tf_2(x_Q)| > 2c \inf_{y \in Q}M_{\gamma,r}  f(y)\}|
\\
&\le |\{z \in Q: |Tf_2(z) - Tf_2(x_Q)| > c_2 \inf_{y \in Q}M_{\gamma, r} f(y)\}|
\\
&\qquad\qquad + |\{z \in Q: |Tf_1(z)| > c_1 \inf_{y \in Q}M_{\gamma, r} f(y)\}|
\\
&<{s}|Q|.
\end{align*}
Whence
\[\inf_{c'} \; \inf  \{\alpha \ge  0: |\{z \in Q: |Tf(z) - {c'}| > \alpha\}| < s|Q| \}\le c \inf_{y \in Q}M_{\gamma,r} f(y),\]
and consequently, since this holds for all $Q\subset Q_0$, $x\in Q$,
\[M^{\sharp}_{0,s,Q_0}Tf(x)\le c \sup_{ {x\in Q,  Q\subset Q_0}} \inf_{y \in Q}M_{\gamma,r} f(y)\,.\]

The proof is thus complete.
\end{proof}

Now, when we also have that $T(1)=0$, under appropriate conditions
on $\omega$  it follows that
$ M^{\sharp}_{0,s,Q_0}(Tf)(x) \le  c\, \sup_{x \in Q, Q \subset Q_0} \inf _{y \in Q} M^{\sharp}_{\gamma,r} f(y)$,
where $M_{\gamma,  {A}}^\sharp f(x)$ has the expected definition \cite{PT}.

That  $M_{\gamma,r}$ is relevant on the right-hand side of (2.4) for all $r$,
 $1\le r <\infty$, is clear from (2.8) above, and it is useful when $T$ is not known to
 be of weak-type $(1,1/(1-\gamma))$.
 Also, there are   operators of weak-type ($1,1/(1-\gamma))$  where $M_{\gamma,r}$ is
 necessary on the right-hand side
of (2.7), and hence on the right-hand side of (2.4),  for $1<r<\infty$.
These are the  convolution fractional operators   of Dini type, i.e.,  $k(x) = {\Omega(x')}/|x|^{n(1-\gamma)}$, $x \ne 0$, where
 $\Omega$ is a function on $S^{n-1}$ that satisfies $\int_{S^{n-1}}\Omega(x')\,dx'=0$ and
 an $L^{r'}$-Dini condition  for some $1\le r'\le \infty$, \cite{ChWW, DLu}.
Because of their similarity with the
fractional type operators with kernels satisfying H\"ormander-type conditions
 considered in Theorem 2.2 below, the
analysis of this case is omitted.

Now, in the latter case we have
\begin{theorem} Let $T$ be a fractional  integral operator of weak-type $(1,1/1-\gamma )$ such that
 for a Young function $A$, every cube $Q$, and $u,v\in Q$,
\[  \sum^\infty_{m=1} |2^{m+1}Q|^{1-\gamma} \| \indicator{2^{m+1}Q\setminus  2^{m}Q}( k (u, \cdot)- k(v,\cdot))\|_{L^{{A}}(2^{m+1}Q)}\le c_A<\infty\,.\]

Then, with $\overline A$ the conjugate Young function to $A$, $c$ independent of $x$, $Q_0$, and $f$, we have
\begin{equation}
 M^{\sharp}_{0,s,Q_0}   (Tf   )(x) \le c \sup_{{x\in Q \,,  Q \subset Q_0}}\inf_{y \in Q}M_{\gamma, \overline{A}}f(y)\,.
 \end{equation}
\end{theorem}

The idea of the proof is essentially that of   \cite{BLRiveros, River}, where $T$ is assumed to be
of convolution type and then the stronger conclusion
\[M^\sharp (|Tf|^\delta)(x)^{1/\delta}\le c\,  M_{\gamma, \overline{A}} f(x) \]
 holds with $\delta$ sufficiently small, or  that of the proof of Theorem 4.3 in \cite{PT}.

 Finally, we consider   integral operators with  homogeneous kernels  defined    as follows \cite{RiverUrci}.
If $A_1, \ldots, A_m$ are invertible matrices such that $A_{k}-A_{k'}$ is invertible
for $k\ne k'$, $1\le k,k'\le m$, and   $\gamma_i > 0$ for all $i$ with $\gamma_1 + \cdots + \gamma_m = n(1 -\gamma)>0$,
then
\begin{equation}
T  f(x) = \int_{\R^n} |x - A_1y|^{-\gamma_1} \cdots |x - A_m y|^{-\gamma_m}f(y)\,dy\,.
\end{equation}

For these operators we have,
\begin{theorem}
For $T$ defined as in {\rm{(2.10)}}, any cube $Q_0 \subset \R^n$, and $x\in Q_0$, we have
\begin{equation}
M^{\sharp}_{0,s,Q_0}  (Tf )(x)  \le c \sum_{i=1}^m \sup_{ {x\in Q \,,Q \subset Q_0}}\inf_{y \in  {Q}} M_{\gamma}
f(A_i^{-1}y)
\,.
\end{equation}
\end{theorem}

Since   by Theorem 3.2 in \cite{RiverUrci}, $T$ is of weak-type $(1,1/1-\gamma)$,
the proof follows along the lines of Theorem 2.1 in  \cite {RiverUrci} or the proof of
 Theorem 4.4 in \cite{PT}.
  To obtain the weighted estimates below, one can of course  use the full strength of the result in \cite{RiverUrci}, namely,
\[M^\sharp(|Tf|^\delta)(x)^{1/\delta}\le c\, \sum_{i=1}^m M_\gamma  f(A_i^{-1} x)\,,
\]
where $0<\delta<1$.

\section{Local Weighted Estimates}

In this section we consider the control of a weighted local mean of a fractional type operator by
the weighted local mean of an appropriate fractional maximal function.
We say that the weights $(w,v)$ satisfy condition $F$ provided
there exist positive constants $c, \alpha, \beta$ with $0 < \alpha <
 1$, such that for any cube $Q$ and  measurable subset $E$ of $Q$ with $|E| \le  \alpha |Q|$,
\begin{equation}
\int_E w(x)\,dx \le  c \Big ( \frac{|E|}{|Q|} \Big )^{\beta}
\int_{Q \setminus E}v(x)\, dx\,.
\end{equation}

Fujii observed that if $v=w$,
(3.1) is equivalent to the $A_{\infty}$ condition for $w$;
he also gave a simple example of a pair $(w , v)$ that satisfy condition $F$
so that neither of them is an $A_\infty$ weight   and no $A_\infty$ weight
can be inserted between them \cite{Fujii1989}.

Now, if $w$ is in weak $A_\infty$, or more generally
  in the Muckenhoupt class  $C_p$,  then $(w, Mw)$ satisfy condition $F$.
On the other hand, from   (3.4) below  it follows that in particular
for some weight $w$,  $(w,Mw)$ do not satisfy
condition $F$. Along these lines,
 for any weight $w$ and $1<r<\infty$,  $  ( w, M_r w   )$ satisfy condition $F$.

The weighted local estimate is then,

\begin{theorem}
Let $T$ be a fractional type operator that satisfies the assumptions of Theorem 2.1,
 $\Phi$   that satisfies condition $C$, and $(w,v)$  weights on $\R^n$ satisfying condition $F$.
Then there exists  $1/2 < t < 1$ such that
\begin{equation}
\int_{Q_0}\Phi  (|Tf(x) - m_{Tf}(t,Q_0)|   )\,w(x)\,dx
\le  c\int_{Q_0}\Phi (M_{\gamma,r}f(x) )\, v (x)\,dx\,.
\end{equation}

Furthermore, if $f$ is such that  $m_{Tf}(t,Q_0) \to 0$ as $Q_0 \to \mathbb{R}^n$, then
\begin{equation}
\int_{\R^n}\Phi(|Tf(x)| )\,w(x)\,dx \le c\int_{\R^n}\Phi (M_{\gamma,r}f(x) )\,v(x)\,dx\,.
\end{equation}

Moreover, if $\Phi$ is concave or $\,\Phi(u)=u$, then {\rm{(3.2)}}  and {\rm{(3.3)}}   hold with $v(x)=Mw(x)$.
\end{theorem}
\begin{proof}
By Theorem 3.1 in \cite{PT} there exist $0 < s \le 1/2$ and  $1/2 < t < 1-s$ such that
\begin{equation*}
\int_{Q_0}\Phi  (|Tf(x) - m_{Tf}(t,Q_0)|   )\,w(x)\,dx
\le  c\int_{Q_0}\Phi (M^{\sharp}_{0,s,Q_0} Tf(x) )\, v (x)\,dx\,,
\end{equation*}
and by Theorem 2.1 above,
\begin{equation*}
\int_{Q_0}\Phi (M^{\sharp}_{0,s,Q_0} Tf(x) )\, v (x)\,dx
\le  c\int_{Q_0}\Phi (M_{\gamma,r}f(x) )\, v (x)\,dx\,.
\end{equation*}
Hence, combining these estimates (3.2) holds, and  (3.3) follows immediately from Fatou's lemma.

The  conclusion for $\Phi$ concave follows from Theorem 3.2  in \cite{PT}.
\end{proof}

As Lerner observed, $\lim_{Q_0\to\R^n}m_g(t,Q_0)=0$ if $g^*(+\infty) = 0$,
where  $g^*$ denotes the nonincreasing rearrangement of $g$,
which in turn holds if and only if
$|\{x \in \R^n : |g(x)| > \alpha\}| < \infty$
for all $\alpha > 0$, \cite{L}.

Now, extending a result of Adams \cite{Ad1}, P\'erez  proved in \cite{CZSobo}
a sharp version of (3.3) for the Riesz potentials. Indeed, Theorem 1.1 (B) there asserts
that for an arbitrary weight $w$ and $1<q<\infty$,
\begin{equation}
 \int_{\R^n} |I_\gamma  f(x)|^q\,w(x)\,dx\le c\,\int_{\R^n} M_\gamma f(x)^q \,M^{[q]+1}(w)(x) dx\,,
\end{equation}
and that  this estimate is sharp, since $[q]$   cannot be replaced by $[q]-1$ above. Thus,
combining (3.3) with (3.4) it readily follows that for each $k\ge 1$, for some weight $w$, $(w,M^kw)$ do  not satisfy condition $F$.

Also, for  $\Phi$ concave, including $\Phi(u)=u$,  from (3.3) it follows in particular that
\[\int_{\R^n} \Phi(|I_\gamma  f(x)|)\,w(x)\,dx\le c\,\int_{\R^n} \Phi(M_\gamma f (x)) \,Mw(x) dx\,, \]
which complements estimate (13) in  \cite{CZSobo}.

P\'erez also addressed in Theorem 1.1 (A) in \cite{CZSobo} the question of estimating
the integral involving the fractional maximal function in the right-hand side of (3.3)
by an integral involving $|f|$ and a posibly larger weight in the right-hand side of (3.3)
as do Bernardis et al in \cite{BLRiveros}.
There is a vast literature of results of this nature, pioneered by Sawyer's work \cite{Sawyer1982}. The results of
Harboure et al \cite{Harboure} are also relevant. Here we prove the result directly in the next section.

\section{Two-weight estimates for fractional type operators}
In this section we consider two-weight,  $L^p_v-L^q_w$ estimates with $1<p < q<\infty$  that apply directly to a
 fractional  type operator  and where the control exerted by a fractional maximal function is not  apparent.

First recall the definition of the classes $B_p$
and $B_{\alpha, p}$. The latter class was introduced by Cruz-Uribe and Moen \cite{CUKM} and
for $0 < \alpha < 1$ and $1 < p < 1/\alpha$, it consists  of those Young functions $A$
such that with $1/q = 1/p - \alpha$,
\[
\|A\|_{\alpha,p} = \Big(\int_c^\infty \frac{A(t)^{q/p}}{t^q}\, \frac{dt}{t}\Big)^{1/q} < \infty\,.
\]
When $\alpha=0$ this reduces to the $B_p$ class of P\'erez.  The result of
interest to us, Theorem 3.3 in \cite{CUKM}, is that for $A \in  B_{\alpha,p}$ the fractional maximal
function $ M_{\alpha,A}f(x)$
maps $L^p(\R^n)$ continuously into $L^q(\R^n)$ with norm not exceeding $  c\,\|A\|_{\alpha,p}$. This result   also
holds for $\alpha=0$, i.e., the $B_p$ classes  \cite{Perez1990}.

We  rely on the following result of P\'erez,  Theorem 2.11 in \cite{CPIUMJ} or
Theorem 3.5 in \cite{GCM}.
 Let $p, q$ with $1 < p < q < \infty$, and  $0 <\gamma < 1$. Let $(w, v)$  be a pair
of weights such that for every cube $Q$,
\begin{equation}
  |Q|^{\gamma }\, |Q|^{1/q-1/p} \| w^{1/q}\|_{L^q(Q)}\, \| v^{-1/p}\|_{L^B(Q)} \le  c\,,
 \end{equation}
where $B$ is a Young function with $\overline B\in B_p$.
Then, if $f \in L^p(v)$,
\[\Big(\int_{\R^n} M_\gamma f(x)^q\, w(x)\, dx\Big)^{1/q}
\le  c\, \Big( \int_{\R^n} |f(x)|^p\, v(x)\, dx\Big)^{1/p} .\]

Before we proceed to prove our next theorem, we need an extension of a property given in
Lemma 4.8 in [24] and the comments that follow it.

\begin{lem}
Let $T$ be a fractional  type operator defined by {\rm{(4.2)}}.
If $T$ satisfies the assumptions of Theorem 2.1 with $1\le r<\infty$, and in that case
 \[\lambda_m=  \omega (c_n/{2^m})\,,\quad m\ge 1\,,\]
 or $T$ satisfies the assumptions  of Theorem 2.2 and in that case
 \[ \lambda_m = \sup_{u,v\in Q}
|2^{m+1}Q|^{1-\gamma} \,  \| \indicator{2^{m+1}Q\setminus  2^{m}Q} ( k (u, \cdot)- k(v,\cdot) \big )\|_{L^{{r'}}(2^{m+1}Q)}\,,
\]
then   $\sum_m\lambda_m<\infty$, and if   $Q$ is a cube of $\R^n$, with $m^\sharp_{Tf}$ as in {\rm{(2.1)}}, we have
\begin{equation}
m^\sharp_{Tf}(1-s, Q) \le
  c\,  \sum_{m=1}^{\infty} \lambda_m  \,  |2^m Q|^\gamma \, \Big(
 \frac{1}{|2^{m}Q| } \int_{2^m Q} |f(y) |^r  \, dy\Big)^{1/r}\,.
 \end{equation}

\end{lem}

\begin{proof}
Fix $Q$,  let $x\in Q$, and put $f=f_1+f_2$ where $f_1=f\indicator{2Q} $.
We claim that there exist constants $c_1,c_2>0$ independent of $f$ and $Q$ such that
\begin{equation}
|\{z \in Q: |Tf_1(z)| >  c_1\, I \}| < s\,|Q|\,,
\end{equation}
and
\begin{equation}
  \|Tf_2 - Tf_2(x_Q) \|_{L^{\infty}(Q)} \le  c_{2} \, I\,,
\end{equation}
where
\[ I= \sum_{m=1}^{\infty} \lambda_m   \, |2^{m}Q|^\gamma \Big(\frac{1}{|2^{m}Q|} \int_{2^m Q} |f(y) |^r  \, dy\Big)^{1/r} \,.
\]

First,  in the  case of Theorem 2.1, $T$ is of weak-type  $(r, r/(1-\gamma r))$ and as in (2.8)
 we have that for any $\lambda > 0$,
   \begin{align*}
{\lambda^{r/(1-\gamma r)}} | \{z \in Q: &|Tf_1(z)| > \lambda\}|\\
&\le   {c} \,\Big( |2Q|^{\gamma} \Big(\frac1{| 2Q|}
  \int_{2Q}|f(y)  |^r\,dy \Big)^{1/r} \Big)^{{r/(1-\gamma r)}} |Q|\,.
\end{align*}
In the case of  Theorem 2.2, $T$ is of weak-type  $(1, 1/(1-\gamma )$ and as in (2.8) by H\"older's inequality
 we have that for any $\lambda > 0$,
   \begin{align*}
{\lambda^{1/(1-\gamma )}} | \{z \in Q: &|Tf_1(z)| > \lambda\}|\\
&\le   {c} \,\Big( |2Q|^{\gamma} \Big(\frac1{| 2Q|}
  \int_{2Q}|f(y)  |^r\,dy \Big)^{1/r} \Big)^{{1/(1-\gamma )}} |Q|\,.
\end{align*}
Thus, in both cases (4.3) holds.

Next, when $T$ satisfies the assumptions
of Theorem 2.2, (4.4) holds automatically.
And,  if  $T$ satisfies the assumptions
of Theorem 2.1,   as in (2.7)  for any $z\in Q$,
\begin{equation*} |Tf_2(z) - Tf_2(x_Q)| \le c\,
\sum_{m=1}^{\infty} \omega (c_n/2^m )\, \frac{1}{|2^m Q|^{1-\gamma}}\int_{2^mQ} |f(y)  |\,dy
\end{equation*}
and, therefore, (4.4) holds by H\"older's inequality.

Then, in either case, with $c > \max \{c_1, c_2\}$, as in the proof of Theorem 2.1, (4.3) and (4.4) give
\begin{equation*}
|\{z \in Q : |Tf(z) - Tf_2(x_Q)| > 2c \, I\}|<{s}|Q|\,,
\end{equation*}
and therefore for all $Q$,  by (2.1),
\begin{align*}m^\sharp_{Tf}(1-s, Q) &=\inf_{c'} \; \inf \{\alpha \ge  0: |\{z \in Q: |Tf(z) - {c'}| > \alpha\}| < s|Q|  \}\\
&\le c\, \sum_{m=1}^{\infty} \lambda_m  \,|2^{m}Q|^\gamma \Big( \frac{1}{|2^{m}Q|} \int_{2^m Q} |f(y) |^r\, dy\Big)^{1/r}\,,
\end{align*}
and we have finished.
\end{proof}

Note that Lemma 4.1 also applies to the convolution fractional  type operators of Dini type.
In that case, using Lemma 1 in \cite{DLu}, the $\alpha_m$ can be estimated in terms
of the $\omega_{r'}$ modulus of continuity of the kernel $k$.

Now the main result.

\begin{theorem} Let $T$ be a fractional  type operator
 that satisfies the assumptions of Theorem 2.1 with $1\le r<\infty$  or  Theorem 2.2 with the
 Young function $t^{r'}$ there and $1\le r<\infty$.
Let $\gamma r<1$, and $r<p<q<\infty$, define $0<\alpha<1$ by the relation
$\alpha=1/p-1/q$, and let $\alpha_1,\alpha_2\ge 0$ be such that $\alpha=\alpha_1+\alpha_2$. Further, suppose that
the Young functions   $A$, $B$ are so that
 $\overline{A} \in B_{( q/r)'}\cap B_{q'}^{\alpha_2}$ and $\overline{B} \in B_{p/r}^{\alpha_1 r}$,
and  $w$ and $v$  weights such that for all cubes $Q$,
\begin{equation}
\sup_Q |Q|^{\gamma r} |Q|^{ r/q -r/p}\,\|w^{r/q}\|_{L^A(Q)}\, \|v^{-r/p}\|_{L^B(Q)}\le c<  \infty \,.
\end{equation}

Then, if $\lambda_m$ is defined as in Lemma 4.1 satisfies
\[\sum_{m=1}^{\infty} \lambda_m 2^{mn   /q} < \infty\,,\]
we have
\begin{equation}
\Big(\int_{\R^n}   | Tf(x)   |^q\, w(x) \, dx\Big)^{1/q}
 \le  c \Big(\int_{\R^n} |f(x)|^p\, v(x)\, dx\Big)^{1/p}
\end{equation}
for those $f$ such that $\lim_{Q_0\to\R^n}m_{Tf}(t, Q_0)=0$.
\end{theorem}

\begin{proof}
We begin by considering the local version of (4.6).
Fix a cube $Q_0$ and note that by the  local median decomposition discussed
in Theorem 6.5 in \cite{PT}
there exists a family $\{Q^v_j\}$ of dyadic subcubes of $Q_0$ such that if $\widehat{Q^v_j}$
denotes the dyadic parent of $Q^v_j$, we have
\begin{align*}
|Tf(x) &- m_{Tf}(t,Q_0)|\\
& \le  8\,M^{\sharp}_{0,s,Q_0}(Tf)(x) +  c  \sum_{v,j} m^\sharp_{Tf}   (1-(1-t)/2^n,\widehat{Q^v_j}  )\indicator{Q^v_j}(x) \,.
\end{align*}

Therefore to estimate the $L^q_w(Q_0)$ norm of $Tf(x) - m_{Tf}(t,Q_0)$ it suffices
to estimate the norm of each summand above separately.
Since by Theorem 2.1 or Theorem 2.2 we have
\[M^{\sharp}_{0,s,Q_0}  (Tf )(x)\le c\, M_{\gamma,r} f(x)=c\, M_{\gamma r}(|f|^r)(x)^{1/r} \,,\]
  the first term above can be estimated by
\begin{equation}
\| M_{\gamma r}   (|f|^r )^{1/r} \|_{L^q_w}=  \|M_{\gamma r}(| f|^r)\|_{L^{q/r}_w}^{1/r}\,.
\end{equation}

Now, since $\overline A\in B_{(q/r)'}$, by  (6.17) in \cite{PT},
\[
\|w^{r/q}\|_{L^{q/r}(Q)}   \le c\, \|w^{r/q}\|_{L^A(Q)}
\]
for all cubes $Q$, and therefore
(4.5) implies (4.1) with indices $p/r$ and $q/r $ there, corresponding to the value $\gamma r$.
Thus $M_{\gamma r}$  maps
$L^{p/r}_v$ continuously into $L^{q/r}_w$ and therefore (4.7) is bounded by
\[\| M_{\gamma r} (|f|^r)\|_{L^{q/r}_w}^{1/r}\le c\,\|\, |f|^r \|_{L^{p/r}_v}^{1/r}= \|f\|_{L^p_v}\,.\]

Hence,
\[ \|M^{\sharp}_{0,s,Q_0}  (Tf )\|_{L^q_w}\le c\, \| f\|_{L^{p}_v}\,.
\]

Next  note that by a purely  geometric argument, if $Q$ is any of the cubes $Q^v_j$, there is a dimensional constant $c$
 such that
\begin{align}
\sum_{m=1}^{\infty} \lambda_m \,|2^m \widehat{Q}|^{\gamma} \Big(\frac{1}{|2^m \widehat{Q}|} &\int_{2^m \widehat{Q}}
|f(y)|^r \, dy\Big)^{1/r} \notag \\
&\le c \sum_{m=1}^{\infty}\lambda_m\, |2^m {Q}|^{\gamma} \Big(\frac{1}{|2^m Q|}\int_{2^m Q} |f(y)|^r\,dy \Big)^{1/r}\,.
\end{align}

To    estimate the norm of the sum by duality, let $h$ be supported in $Q_0$ with  $\|h\|_{L^{q'}_v(Q_0)} = 1$
and note that by (4.2) and (4.8),
\begin{align}
&\ \ \int_{Q_0} \Big(  \sum_{v,j} m^\sharp_{Tf}  (1-(1-t)/2^n,\widehat{Q^v_j} ) \indicator{Q^v_j}(x) \Big ) w (x)^{1/q} h(x)\,dx
\notag\\
& \le c\sum_m\lambda_m \sum_{v,j}  |2^{m}Q_j^v|^\gamma \Big( \frac{1}{|2^{m}Q_j^v|}
\int_{2^m Q_j^v} |f(y) |^r\, dy \Big)^{1/r}\!\! \! \int_{Q_j^v} w (x)^{1/q} h(x)\,dx .
\end{align}

We consider each term in the inner sum of (4.9) separately.
First,
let $D$ be the Young function defined by $D(t)={\overline{B}}(t^r)$, and note
that since $\|g\|_{L^D(Q)}= \| |g|^r\|_{L^{\overline B}(Q)}^{1/r}$,
 by H\"older's inequality for the conjugate Young functions $B,\overline B$,
\begin{align*}
\Big ( \frac{1}{|2^{m}Q_j^v|}  \int_{2^m Q_j^v} &|f(y) |^r \, dy \Big )^{1/r}\\
&=
\Big ( \frac{1}{|2^{m}Q_j^v|} \int_{2^m Q_j^v} |f(y) |^r \,v(y)^{r/p} \,v(y)^{-r/p} \, dy \Big )^{1/r}\\
&\le 2\, \big(   \||f|^r v^{r/p} \|_{L^{\overline{B}}(2^mQ^v_j)} \, \| v^{-r/p} \|_{L^{B}(2^mQ^v_j)} \big)^{1/r} \,,
\\
& = 2\, \| f v^{1/p} \|_{L^{D}(2^mQ^v_j)} \,  \| v^{-r/p} \|_{L^{B}(2^mQ^v_j)}^{1/r}\,.
\end{align*}

Next, let $C$ be the Young function defined by $C(t)= A(t^r)$  and note
that as above,  by H\"older's inequality for the conjugate Young functions $C,\overline C$,
 \begin{align*}
 \int_{Q_j^v} w (x)^{1/q} \, h(x)\,dx &\le
 2^{mn} |Q_j^v|\, \frac{1}{|2^{m}Q_j^v|} \int_{2^m Q_j^v} w(x)^{1/q} \, h(x)\indicator{Q_j^v}(x)\,dx
\\
&\le 2\cdot 2^{mn}\,  \|w^{1/q}\|_{L^{C}(2^m Q^v_j)} \|h \indicator{Q^v_j}\|_{L^{\overline{C}}(2^m Q^v_j)}\,|Q_j^v|
\\
&\le 2\cdot 2^{mn}\, \|w^{r/q}\|_{L^{A}(2^m Q^v_j)}^{1/r} \|h \indicator{Q^v_j}\|_{L^{\overline{C}}(2^m Q^v_j)}\,|Q_j^v|\,.
\end{align*}

Moreover, since for each $\lambda>1$ and each cube $Q$ we have
\[  \|g \indicator{Q}  \|_{L^{\overline{C}}(\lambda  Q)} \le  \|g   \|_{L^{\overline{C} /\lambda^n}(Q)}\,,
\]
it follows that
\[ \int_{Q_j^v} w (x)^{1/q} \, h(x)\,dx \le 2\cdot
2^{mn} \, \|w^{r/q}\|_{L^{A}(2^m Q^v_j)}^{1/r}\, \|h \|_{L^{{\overline{C}/2^{mn}}}(  Q^v_j)}\, |Q_j^v|\,.\]

Therefore, since by (4.5)   with $1/p-1/q=\alpha$,
\[ |2^mQ^v_j|^\gamma \, \|w^{r/q}\|_{L^{A}(2^m Q^v_j)}^{1/r}\,\| v^{-r/p} \|_{L^{B}(2^mQ^v_j)}^{1/r}\le c\,  |2^mQ^v_j|^{\alpha}\,,
\]
each term in the inner sum of (4.9) is bounded  by
\[ c\,2^{mn} \, | 2^m Q^v_j|^{\alpha } \,\| f\,  v^{1/p} \|_{L^{D}(2^mQ^v_j)}
\,  \|h \|_{L^{{\overline{C}/2^{mn}}}(  Q^v_j)}\,|Q_j^v|\,,
\]
and consequently the sum itself does not exceed
\begin{equation}
c\, \sum_{m=1}^\infty \lambda_m
 2^{mn} \, \sum_{v,j} | 2^m Q^v_j|^{\alpha}\, \|  f v^{1/p} \|_{L^{D}(2^m Q^v_j)}\,
 \|h    \|_{L^{\overline{C}/2^{mn}}(  Q^v_j)}\,  |Q_j^v|\,.
\end{equation}

 Let $F^v_j = Q^v_j \setminus \Omega^{v+1}$;  then
the $F^v_j$ are pairwise disjoint and $|F^v_j| \ge  c |Q^v_j|$,
where $c$ depends on $s$ and $t$  but is independent of $v$ and $j$.
Now,  since $\alpha=\alpha_1+\alpha_2$, the innermost sum in (4.10) is bounded by
\[ J= c\, \sum_{v,j}   | 2^m Q^v_j|^{\alpha_1}\,
 \| f v^{1/p}  \|_{L^{D}(2^m  Q^v_j)}  \,  | 2^m Q^v_j|^{\alpha_2}\, \|h    \|_{L^{\overline{C}/2^{mn}}(  Q^v_j)} |F^v_j|\,,
\]
and, since
\[|2^m Q^v_j|^{\alpha_1}\, \|  f v^{1/p} \|_{L^{D}(2^m Q^v_j)}\le \inf_{x\in F_j^v}
M_{\alpha_1, D} (f v^{1/p})(x)\]
and similarly
\[ |2^m Q^v_j|^{\alpha_2}\,\|h    \|_{L^{\overline{C}/2^{mn}}(  Q^v_j)}\le
\inf_{x\in F_j^v} M_{\alpha_2,{\overline{C}/2^{mn}}}h(x)\,,
\]
we have that
\begin{align*}
   J &\le  c   \sum_{v,j} \int_{F^v_j}  M_{\alpha_1, D}  (f v^{1/p} )(x) \, M_{\alpha_2,\overline{C}/2^{mn}}\, h(x) \, dx
\\
&\le  c   \int_{Q_0}M_{\alpha_1,D}  (f v^{1/p} )(x) \, M_{\alpha_2,\overline{C}/2^{mn}}\,h(x) \, dx\,.
\end{align*}

Now pick $s_1,s_2$ such that
\begin{equation}
 1/p-\alpha_1 =1/s_1\,,\quad{\text{and}}\quad 1/q'-\alpha_2 =1/s_2\,.
\end{equation}
Since
\[ 1/s_1+1/s_2 =1/p-\alpha_1+ 1-1/q -\alpha_2 = 1/p-\alpha -1/q +1=1\,,
\]
$s_1,s_2$ are conjugate exponents, and, therefore, by H\"older's inequality,
\begin{align*} \int_{Q_0}M_{\alpha_1,D}  (f v^{1/p} )(x)  &M_{\alpha_2,\overline{C}/2^{mn}}\,h(x) \, dx\\
&\le
\| M_{\alpha_1,D}  (f v^{1/p} )\|_{L^{s_1}}  \| M_{\alpha_2,\overline{C}/2^{mn}}\,h\|_{L^{s_2}} .
\end{align*}

Now,  by (iii) and (iv) in
Proposition 6.1 in \cite{PT} ,  $D \in B_{p}^{\alpha_1} $ and  $\overline{C} \in B_{q'}^{\alpha_2}$, respectively,
and, therefore, by  Theorem 3.3 in \cite{CUKM},
\[ \int_{Q_0}M_{\alpha_1,D}  (f v^{1/p} )(x)   M_{\alpha_2,\overline{C}/2^{mn}}\,h(x) \, dx\le
 c\, \|f v^{1/p} \|_{L^p} \,  2^{-{mn/q'}}\, {\|h\|_{L^{q'}(Q_0)}}   \,,
 \]
and   the right-hand side of (4.10) is bounded by
\[ c \,\Big ( \sum_{m=1}^{\infty} \lambda_m 2^{mn(1 - 1/q')} \Big )   \|f  \|_{L^p_v} \le c\, \|f  \|_{L^p_v}\,.
\]
Hence,  combining the above estimates,
\[ \| Tf - m_{Tf}(t,Q_0) \|_{L^q_w(Q_0)}\le c \,\|f  \|_{L^p_v}\,.
\]

Finally,   by Fatou's lemma,
(4.6) follows for functions $f$ such that $m_{ Tf}(t,Q_0) \to 0$ as $Q_0 \to \R^n$.
\end{proof}

\section{Orlicz-Morrey Spaces}
Given a Young function $\Phi$  let
\[ \| f\|_{L^\Phi_Q}=\inf\Big\{\lambda>0: \int_Q \Phi\Big(\frac{|f(y)|}{\lambda}\Big)\,dy=1\Big\}\,;
\]
note that this definition does not coincide with the one  given in (2.2) but in view of (5.3) below
it is more natural in our setting.

Now, for   a  positive continuous (or more generally measurable)   function $\phi(x,t)$ on $\R^n\times \R^+$
such that for each $x\in\R^n$,  $\phi(x,t) $ is decreasing for $t$ in $[0,\infty)$, and $\phi(x,0)=\infty$
for all $x\in\R^n$, with $Q=Q(x,l)$, let
\[ \|f\|_{\mathcal M^{\Phi, \phi} }=\sup_{Q\subset \R^n} \frac1{\phi(x, l)}\,
\Phi^{-1}(1/|Q|)\,\| f\|_{L^\Phi_Q}\,.\]

Although a priori the functions $\Phi$ and $\phi$ are unrelated, even in the simplest case
there are some limitations \cite{SawST}.
Also note that if  $\Phi(t)=t^p$  and $\phi(x,t)=t^{(\lambda- n)/p}$,
$0<\lambda <n$, then $\mathcal M^{\Phi,\phi} =\mathcal M^{p,\lambda}$,
the familiar Morrey space. Similar definitions, also coinciding with  $\mathcal M^{p,\lambda}$ when $\Phi$ is a power, are
given  in \cite{Nakai} and \cite{SawST}.

As for the Campanato spaces $\mathcal L^{\Phi, \phi} $, consider the seminorms
\[ \|f\|_{\mathcal L^{\Phi, \phi} } =\sup_{Q\subset \R^n} \frac1{\phi(x, l)}\,
\Phi^{-1}(1/|Q|)\,\inf_c \| f-c\|_{L^\Phi_Q}\,.\]

Now, as pointed out in (1.9),  if $T$ is a fractional  type operator that satisfies the
conditions of Theorem 2.1 with $r=1$, from (3.3)
it readily follows that for  every Young function $\Psi$ and appropriate $\psi$,
\[ \|Tf\|_{\mathcal L^{\Psi, \psi}}\le c\,\|M_\gamma f\|_{\mathcal M^{\Psi, \psi}}\,.\]

We are therefore led to explore the continuity properties of $M_\gamma$ in the Orlicz-Morrey spaces. We
assume that the Young functions $\Phi, \Psi$ satisfy the relation
\begin{equation}
  \Psi^{-1}(t)=t^{-\gamma} \Phi^{-1}(t)\,,
\end{equation}
which gives that  $M_\gamma$, and also $I_\gamma$, map $L^\Phi(\R^n)$ continuously into $L^\Psi(\R^n)$ \cite{TorInt}.

Observe that for a cube $Q$, by H\"older's inequality for the conjugate functions $\Phi,\overline \Phi$,
\[\int_{Q}|f(z)|\,dz\le 2\, \|\indicator{Q}\|_{L^{\overline \Phi (Q)}}\,  \|f\|_{L^\Phi (Q)}\,, \]
which, by the relations
$ \|\indicator{Q}\|_{L^{\overline \Phi(Q)}} \sim 1/{\overline \Phi}\phantom{}^{-1}(1/|Q|)$ and
 $\Phi^{-1}(t){\overline \Phi}\phantom{}^{-1}(t)\sim t$, gives that
\begin{equation}
\frac1{|Q|} \int_{Q}|f(z)|\,dz\le c\,   \Phi^{-1}(1/|Q|)\,\|f\|_{L^\Phi_Q}\,.
\end{equation}

Hence, in particular if  (5.1) holds we have
\begin{equation}
\frac1{|Q|^{1-\gamma}} \int_{Q}|f(z)|\,dz\le c\,   \Psi^{-1}(1/|Q|)\,\|f\|_{L^\Phi_Q}\,.
\end{equation}

Along the lines of \cite{GulShu} we begin by proving a preliminary result.
\begin{prop}
Suppose $f$ is locally integrable and $0\le \gamma<1$.
 Then there exist dimensional constants $c_n,d_n$ with $c_n d_n\ge 1$ such that for an arbitrary cube $Q=Q(x,l)$,
\begin{enumerate}
\item[\rm (i)]
\begin{align*}
\|M_\gamma f  \|_{L^\Psi_Q}
 \le c\, &\|f\|_{L^\Phi_{2Q}}\\
&+ c\, \frac1{{\Psi^{-1}(1/|Q|)}} \sup_{t>c_n d_n l }
\Big( \frac1{ |Q(x,t)|^{1-\gamma}}\int_{Q (x, t)} |f(z)|\, dz \Big);
\end{align*}
\item[\rm (ii)]
\begin{equation*}
\| M_\gamma f \|_{L^\Psi_Q}\le c\,\frac1{\Psi^{-1}(1/|Q|)}\,
\Big( \sup_{t\ge  c_n d_n l  } \Psi^{-1}(1/|Q(x,t)|)\| f\|_{L^\Phi_{Q(x, t)}} \Big).
 \end{equation*}
\end{enumerate}
\end{prop}
\begin{proof} Fix a cube $Q$ and let  $f=f_1+f_2$, where $f_1=f\indicator{2Q}$ and $f_2=f-f_1$.
Then as observed above,
\begin{equation}
\|M_\gamma f_1\|_{L^\Psi(Q)}\le c\,\|f_1\|_{L^\Phi(\R^n)}=c\,\|f\|_{L^\Phi(2Q)}\,.
\end{equation}

Next we estimate $M_\gamma  f_2(x')$ for $x'\in Q$.
 Note that if $x' \in Q(y,t)$,
 by purely geometric considerations  there exist dimensional constants $c_n,d_n$ such that
1. If $Q(y,t)\cap (\R^n\setminus 2Q)\ne\emptyset$, then $t>c_n l $, and 2.
For any $x\in Q$, $Q(y,t)\cap (\R^n\setminus 2Q)\subset Q(x, d_n t)$ and $c_nd_n\ge 1$.

Hence, for $x'\in Q\cap Q(y,t)$,  we have that $t>c_n l $ and
\begin{align*}
\frac1{|Q(y,t)|^{1-\gamma}}\int_{Q(y,t)}|f_2(z)|\,dz &= \frac1{|Q(y,t)|^{1-\gamma}}\int_{Q(y,t)\cap \,(\R^n\setminus  2Q)} |f(z)|\,dz
\\
&\le c \, \frac1{|Q(x, d_n t)|^{1-\gamma}}\int_{Q(x, d_n t)}|f(z)|\,dz\,,
\end{align*}
and therefore for any $x'\in Q$,
\begin{align}  M_\gamma f_2(x') &= \sup_{x'\in Q(y,t)} \frac1{|Q(y,t)|^{1-\gamma}}\int_{Q(y,t)}|f_2(z)|\,dz
\notag\\
&\le c\,  \sup_{t> c_n d_n l } \frac1{|Q(x,t)|^{1-\gamma}}\int_{Q(x,t)}|f(z)|\,dz\,.
\end{align}

Now, since for $g\in L^\infty_Q$,
\begin{equation}
\| g\|_{L^\Psi_Q}\le \|\indicator{Q}\|_{L^\Psi} \|g\|_{L^\infty_Q}
\le c\,\frac1{\Psi^{-1}(1/|Q|)}\, \|g\|_{L^\infty_Q}\,,
\end{equation}
from (5.5) and (5.6) it follows that
\begin{equation}
\| M_\gamma f_2\|_{L^\Psi_Q}\le c\,\frac1{\Psi^{-1}(1/|Q|)}\,  \sup_{t> c_n d_n l }
  \frac1{|Q(x,t)|^{1-\gamma}}\int_{Q(x,t)}|f(z)|\,dz\,.
 \end{equation}

Thus, since $\| M_\gamma f\|_{L^\Psi_Q}\le \| M_\gamma f_1\|_{L^\Psi_Q}+\| M_\gamma f_2\|_{L^\Psi_Q}$, (i) follows
combining (5.4) and (5.7).

As for (ii), note that
\begin{align}  \|f\|_{L^\Phi_{2Q}} &\le c\, \frac{1}{\Psi^{-1}(1/|Q|)}
\sup_{t\ge c_n d_n l } \Psi^{\ -1}(1/|Q(x, t)|)   \| f\|_{L^\Phi_{2Q}}
\notag\\
& \le c\, \frac{1}{\Psi^{-1}(1/|Q|)} \sup_{t\ge c_n d_n l }\Big( \Psi^{\ -1}(1/|Q(x,t)|)  \| f\|_{L^\Phi_{Q(x, t)}}
\Big)\,.
\end{align}

Also, by (5.3),
\begin{align}
 \frac{1}{\Psi^{-1}(1/|Q|)} &\Big( \sup_{t\ge  c_n d_n l }  \frac1{|Q(x, t)|^{1-\gamma}}\int_{Q(x, t)} |f(z)|\,dz \Big)
\notag\\
& \le c\, \frac{1}{\Psi^{-1}(1/|Q|)} \Big( \sup_{t\ge  c_n d_n l  } \Psi^{-1}(1/|Q(x,t)|)\| f\|_{L^\Phi_{Q(x, t)}} \Big),
\end{align}
and (ii) follows combining (5.8) and (5.9).
\end{proof}

Concerning the continuity of $M_\gamma$ in the Orlicz-Morrey spaces, in
the spirit of Theorem 4.3 of \cite{GulShu} we have,
\begin{theorem}
Let $0\le \gamma<1$  and suppose that  $\phi,\psi$ satisfy the condition
\[\sup_{r<t<\infty}  t^{n\gamma} \phi(x,t)\le c\, \psi(x,r)\,.
\]
Then $M_\gamma$ maps $M^{\Phi,\phi}$ continuously into $M^{\Psi,\psi}$.
\end{theorem}
\begin{proof}
First  note that by (5.2),
\begin{align}
\Psi^{-1}(1/|Q(x,t)|) \| f\|_{L^\Phi_{Q(x, t)}}
&=  t^{n\gamma} \,\Phi^{-1}(1/|Q(x,t)|) \| f\|_{L^\Phi_{Q(x, t)}}\notag\\
&\le    \Big(  \sup_{t\ge l } t^{n \gamma} \phi(x,t)\Big)\, \|f\|_{M^{\Phi,\phi}}\notag\\
&\le c\,\psi(x,l)\, \|f\|_{M^{\Phi,\phi}}\,.
\end{align}

To estimate
$ \|M_\gamma f\|_{M^{\Psi,\psi}}$
observe that by (ii) in Proposition 5.1 and (5.10), since $c_nd_n\ge 1$, it readily follows that
\begin{align*}
\|M_\gamma f\|_{M^{\Psi,\psi}} &\le \sup_{x\in Q, l >0}  \frac1{\psi(x,l )}
\Big( \sup_{t\ge l } \Psi^{-1}(1/|Q(x,t)|)\| f\|_{L^\Phi_{Q(x, t)}} \Big)\notag
\\
&\le c\, \|f\|_{M^{\Phi,\phi}} \,,
\end{align*}
which completes the proof.
\end{proof}

The above result is a prototype for results of the following nature.
Let $S$ be a  sublinear operator   that maps $L^\Phi(\R^n)$ continuously
into    $L^\Psi(\R^n)$ such that
  for any cube $Q$, if $x\in Q$ and  ${\text{supp}}(f)\subset \R^n\setminus 2Q$, then
\begin{equation}
\ |Sf(x)|\le c\,\int_{\R^n} \frac{|f(y)|}{|x-y|^{n (1-\gamma)}} \,dy\,.
\end{equation}
Such operators are considered for instance in \cite{GulShu0}, and they include the
fractional maximal functions as well as the Riesz potentials.

The reader should have no difficulty proving the following.
\begin{theorem}
Let $0\le \gamma <1$,   $\Phi,\Psi$ be Young functions
so that $\Psi^{-1}(t)=t^{-\gamma}\Phi^{-1}(t)$, and   $\phi(x,t), \psi(x,t)$ positive measurable
decreasing functions such that for all $x\in\R^n$ and $\, l>0$,
\begin{equation*}
{\psi(x, l)} \int_{ l}^\infty \frac1{\phi(x,t)}\, \frac{dt}{t}\le c\,.
\end{equation*}
Then, if $S$ is a sublinear operator that maps $L^\Phi(\R^n)$ continuously
into    $L^\Psi(\R^n)$ and  satisfies {\rm (5.11)},
\[ \|Sf\|_{\mathcal M^{\Psi, \psi}  }\le c\,\|f\|_{\mathcal M^{\Phi, \phi }}\,.
\]
\end{theorem}

\end{document}